\documentclass{amsart}

\usepackage{amsmath}
\usepackage{amsfonts}
\usepackage{amssymb,enumerate}
\usepackage{amsthm,stmaryrd}
\usepackage[all]{xy}
\usepackage{hyperref}
\usepackage{tikz}
\usetikzlibrary{matrix,arrows,decorations.pathmorphing}
\theoremstyle{plain}
\newtheorem{lem}{Lemma}[section]
\newtheorem{cor}[lem]{Corollary}
\newtheorem{prop}[lem]{Proposition}
\newtheorem{thm}[lem]{Theorem}

\theoremstyle{definition}
\newtheorem{defn}[lem]{Definition}
\newtheorem{ex}[lem]{Example}

\newtheorem{fact}[lem]{Fact}
\newtheorem{para}[lem]{}

\newtheorem{remark}[lem]{Remark}
\newtheorem{notation}[lem]{Notation}

\newtheorem{assumption}[lem]{Assumption}

\renewenvironment{proof}{\vspace{1ex}\noindent{\textbf{Proof:}}\hspace{0.5em}}
{\hfill\qed\vspace{1ex}}

\newcommand{\cat}[1]{\mathcal{#1}}

\newcommand{\cata}{\cat{A}}

\newcommand{\catb}{\cat{B}}

\newcommand{\catbc}{\cat{B}_C}

\newcommand{\ext}{\operatorname{Ext}}	
\newcommand{\rhom}{\mathbf{R}\!\operatorname{Hom}}	
\newcommand{\lotimes}{\otimes^{\mathbf{L}}}
\newcommand{\HH}{\operatorname{H}}
\newcommand{\Hom}{\operatorname{Hom}}

\newcommand{\s}{\mathfrak{S}}
\newcommand{\tor}{\operatorname{Tor}}

\newcommand{\xra}{\xrightarrow}

\renewcommand{\geq}{\geqslant}

\renewcommand{\hom}{\Hom}

\begin{document}

\title{Semidualizing DG Modules over Tensor Products}
\author{Hannah Altmann}
%

\thanks{This material is based on work supported by North Dakota EPSCoR and National Science Foundation Grant EPS-0814442}



\maketitle

\begin{abstract}  We study the existence of nontrivial semidualizing DG modules over tensor products of DG algebras over a field. In particular, this gives a lower bound on the number of semidualizing DG modules over the tensor product. 
\end{abstract}

\makeatletter
\@setabstract
\makeatother


\section{Introduction} 
\begin{assumption} Let $R$ be a commutative, noetherian ring with identity.
\end{assumption}

\noindent Semidualizing modules were introduced by Foxby {\cite{foxby:gmarm}}, while Vasconcelos {\cite{vasconcelos:dtmc}} and Golod {\cite{golod:gdagpi}} rediscovered them independently and applied them in different contexts. A finitely generated $R$-module $C$ is \textit{semidualizing} over $R$ if the homothety map $\chi_C^R:R\to \hom_R(C,C)$ is an isomorphism and $\ext_R^i(C,C)=0$ for all $i>0$. Let $\s_0(R)$ denote the set of isomorphism classes of semidualizing $R$-modules. The size of $\s_0(R)$ measures the severity of the singularity of a ring, specifically how close a ring is to being Gorenstein. If $\s_0(R)$ is large, then $R$ is far from being Gorenstein. If $\s_0(R)$ is small, then $R$ is in a sense close to being Gorenstein. For instance, if $R$ is Gorenstein and local, then $|\s_0(R)|=1$.

Throughout this paper, we use the more general definition of semidualizing DG module. (``DG'' is short for ``Differential Graded''. See Section 2 for relevant background information.) The idea for the definition is essentially from Christensen and Sather-Wagstaff {\cite{christensen:dvke}}; see also {\cite{nasseh:lrfsdc}}. The DG setting has been useful for answering questions about rings. For instance, Nasseh and Sather-Wagstaff {\cite{nasseh:lrfsdc}} were able to answer Vasconcelos' question {\cite[p. 97]{vasconcelos:dtmc}}, showing a local ring has only finitely many isomorphism classes of semidualizing modules.
 
\begin{defn}
Let $A$ be a DG $R$-algebra. A \textit{semidualizing} DG $A$-module is a homologically finite DG $A$-module $C$ that admits a degreewise finite semifree resolution over $A$ such that the homothety morphism $\chi_C^A:A\to \rhom_A(C,C)$ is an isomorphism in the derived category $D(A)$. Let $\s(A)$ denote the set of shift-isomorphism classes of semidualizing DG $A$-modules in $D(A)$. 
\end{defn}

\noindent What follows is the main result of this paper, which is proven in \ref{proof141021a}. The big picture idea here is that the singularity of the ring $A'\otimes_k A''$ is at least as bad as the singularities of both $A'$ and $A''$ combined.

\begin{thm}\label{thm141112b}
 Let $k$ be a field. Let $A'$ and $A''$ be local DG $k$-algebras such that $A'_0$ and $A''_0$ are noetherian. Let $M'\in D^f_b(A')$ and $M''\in D^f_b(A'')$.
  \enumerate
  \item[(a)] One has $M'\otimes_k M''$ is semidualizing over $A'\otimes_k A''$ if and only if $M'$ is semidualizing over $A'$ and $M''$ is semidualizing over $A''$.
  \item[(b)] The map $\psi:\s(A')\times \s(A'')\to \s(A'\otimes_k A'')$ defined by $\psi(C',C'')=C'\otimes_k C''$ is well-defined and injective.
\end{thm}

Note that part (a) is a consequence of the K\"unneth Theorem, properly interpreted. However, parts (b) and (c) use an extension of Foxby and Christensen's Bass classes to the DG setting.

\section{Background} 
For a thorough introduction to DG algebras see any of the following {\cite{avramov:ifr,avramov:dgha,beck:kddga,beck:sgidgca,nasseh:lrfsdc}}. Below is a quick review of a few of the necessary definitions.

\begin{remark} In this paper $R$-complexes are indexed homologically, and $|a|=i$ means $a\in X_i$.
\end{remark}

\begin{defn} A \textit{commutative differential graded algebra over $R$} (``DG $R$-algebra'' for short) is an $R$-complex $A$ equipped with binary operations $\mu^A : A_i\times A_j\to A_{i+j}$ with $ab:= \mu^A(a,b)$ satisfying the following properties:\\
\indent \textbf{associative:} for all $a,b,c\in A$ we have $(ab)c=a(bc)$;\\
\indent \textbf{distributive:} for all $a,b,c\in A$ such that $|a|=|b|$ we have $(a+b)c=ac+bc$ and $c(a+b)=ca+cb$;\\
\indent \textbf{unital:} there is an element $1\in A_0$ such that for all $a\in A$ we have $1a=a$;\\
\indent \textbf{graded commutative:} for all $a,b\in A$ we have $ab=(−1)^{|a||b|}ba$ and $a^2=0$ when $|a|$ is odd;\\
\indent \textbf{positively graded:} $A_i = 0$ for $i<0$; and\\
\indent \textbf{Leibniz Rule:} for all $a,b\in A$ we have $\partial^A_{|a|+|b|}(ab)=\partial^{A}_{|a|}(a)b+(-1)^{|a|}a\partial^{B}_{|b|}(b)$.\\ 
Given a DG $R$-algebra $A$, the \textit{underlying algebra} is the graded commutative $R$-algebra $A^{\natural} =\oplus_{i=0}^{\infty} A_i$. We say that $A$ is \textit{weakly noetherian} if $H_0(A)$ is noetherian and the $H_0(A)$-module $H_i(A)$ is finitely generated for all $i\geq 0$. We say that $A$ is \textit{mildly noetherian} if $A$ is weakly noetherian and $A_0$ is noetherian. We say that $A$ is \textit{local} if it is weakly noetherian, $R$ is local, and the ring $\HH_0(A)$ is a local $R$-algebra. A \textit{morphism} of DG $R$-algebras is a chain map $f:A\to B$ between DG $R$-algebras respecting products and multiplicative identities: $f(aa′)=f(a)f(a′)$ and $f(1)=1$. 
\end{defn}

\begin{assumption}  For the rest of this section $A$ is a DG $R$-algebra and $k$ is a field.
\end{assumption}

\begin{defn} A \textit{differential graded module over $A$}
(``DG $A$-module'' for short) is an $R$-complex $M$ equipped with binary operations
$\mu^M:A_i\times M_j\to M_{i+j}$ with $am:=\mu^M(a,m)$ satisfying the following properties:\\
\indent \textbf{associative:} for all $a,b\in A$ and $m\in M$ we have $(ab)m = a(bm)$\\
\indent \textbf{distributive:} for all $a,b\in A$ and $m,n\in M$ such that $|a|=|b|$ and $|m|=|n|$,
we have $(a+b)m=am+bm$ and $a(m+n)=am+an$;\\
\indent \textbf{unital:} for all $m\in M$ we have $1m=m$;\\
\indent \textbf{Leibniz Rule:} for all $a\in A$ and $m\in M$ we have $\partial_{|a|+|m|}(am)=\partial_{|a|} (a)m+ (−1)^{|a|} a∂_{|m|}(m)$.\\
The \textit{underlying $A^{\natural}$}-module associated to $M$ is the $A^{\natural}$ -module $M^{\natural}=\oplus_{i=-\infty}^{\infty} M_i$. Let $D(A)$ denote the derived category of DG $A$-modules. Isomorphisms in $D(A)$ are identified by the symbol $\simeq$.
\end{defn}

\begin{defn} A DG $A$-module $M$ is \textit{degreewise finite}, denoted $M\in D^f(A)$, if $\HH_i(M)$ is finitely generated over $\HH_0(A)$ for all $i$. We say $M$ is \textit{homologically bounded}, denoted $M\in D_b(A)$, if $\HH_i(M)=0$ for $|i|\gg 0$. Additionally, $M$ is \textit{homologically finite} if $M\in D_b(A)\cap D^f(A)$, i.e., $M\in D^f_b(A)$. We say $M$ is homologically bounded below, denoted $M\in D_{+}(A)$, if $\operatorname{inf}(M)>-\infty$. Similarly, $M$ is homologically bounded above, denoted $M\in D_{-}(A)$, if $\operatorname{sup}(M)<\infty$. 
\end{defn}

\begin{fact} Let $A'$ and $A''$ be DG $R$-algebras. Let $N'$ be a DG $A'$-module and $N''$ be a DG $A''$-module. Then
\begin{enumerate}
 \item[(a)] $A'\otimes_R A''$ is a DG $R$-algebra via the multiplication $(a'\otimes a'')(b'\otimes b'')=(-1)^{|a''||b'|}(a'b')\otimes (a''b'')$, and
 \item[(b)] the complex $N'\otimes_R N''$ is a DG $A'\otimes_R A''$-module via the multiplication $(a'\otimes a'')(n'\otimes n'')=(-1)^{|a''||n'|}(a'n')\otimes (a''n'')$.
\end{enumerate}
\end{fact}

\begin{defn} A \textit{semibasis} for a DG $A$-module $M$ is a set $E=\bigsqcup_{i=0}^{\infty} E^i$ such that $\partial(E^i)\subseteq A E^{i-1}$ for each $i\geq 0$ (we set $A E^{-1}=0)$ and $E$ is a basis of the $A^{\natural}$-module $M^{\natural}$. We say $M$ is \textit{semifree} if it has a semi-basis. A \textit{semifree resolution} of a DG $A$-module $N$ is a quasiisomorphism $F\xra{\simeq} N$ such that $F$ is semi-free over $A$. We say that a DG $A$-module $M$ is \textit{semiprojective} if $\hom_A(M,-)$ respects surjective quasiisomorphisms. A \textit{semiprojective resolution} of a DG $A$-module $N$ is a quasiisomorphism $P\xra{\simeq} N$ such that $P$ is semiprojective over $A$.
\end{defn}

\begin{fact}\label{fact141013c} Let $M$ and $N$ be DG $A$-modules.
\begin{enumerate}
 \item[(a)] There exists a semifree resolution $F\xra{\simeq} M$.
 \item[(b)] If $A$ is weakly noetherian and $M\in D^f_{+}(A)$, then there exists a semifree resolution $F\xra{\simeq} M$ with semibasis $E$ such that $|E\cap M_n|<\infty$. We call such a resolution a ``degreewise finite semifree resolution.''
 \item[(c)] If $F$ is semifree over $A$, then $F$ is semiprojective over $A$.
 \item[(d)] If $M\in D_b(A)$, then $M$ is semifree over $A$ if and only if $M^{\natural}$ is a free graded $A^{\natural}$-module.
\end{enumerate}
\end{fact}

The following notion was defined for dualizing modules by Foxby {\cite{foxby:gdcmr}} and for an arbitrary semidualizing module or complex by Christensen {\cite{christensen:scatac}}. 

\begin{defn} Let $C\in \s(A)$ and $M\in D_b(A)$. Then $M$ is in the \textit{Bass class} $\catb_C(A)$ if the 
natural evaluation morphism $\xi_M^C:C\otimes_A^L \rhom_A(C,M)\to M$ is an isomorphism in $D(A)$ and $\rhom_A(C,M)\in D_b(A)$.

 Notice, if we are working over $R$, with $C\in \s_0(R)$ and $M$ an $R$-module, this translates as follows: $M$ is in the Bass class $\catb_C(R)$ if the natural evaluation homomorphism $\xi:C\otimes_R \hom_R(C,M)$ is an isomorphism and $\ext_R^i(C,M)=0=\tor_i^R(C,\hom_R(C,M))$ for all $i>0$.
\end{defn}
 Similarly, we have the following notions of the Auslander class and derived reflexive DG modules.
 
\begin{defn} Let $C\in \s(A)$ and $M\in D_b(A)$. Then $M$ is in the \textit{Auslander class} $\cata_C(A)$ if the natural morphism $\gamma_M^C:M\to \rhom_A(C,C\otimes_A M)$ is an isomorphism in $D(A)$ and $C\otimes_A^L M\in D_b(A)$.

 If we are working over $R$, with $C\in \s_0(R)$ and $M$ an $R$-module, this translates as follows: $M$ is in the Auslander class $\cata_C(R)$ if the natural map $\gamma_M^C:M\to \hom_R(C,C\otimes_A M)$ is an isomorphism and $\tor_i^R(C,M)=0=\ext_R^i(C,C\otimes_k M)$ for all $i>0$.
\end{defn}

\begin{defn} Let $C\in \s(A)$ and $M\in D_b^f(A)$. Then $M$ is \textit{derived $C$-reflexive} if the natural biduality morphism $\delta_M^C:M\to \rhom_A(\rhom_A(M,C),C)$ is an isomorphism in $D(A)$ and $\rhom_A(M,C)\in D_b^f(A)$.

 If we are working over $R$, where $C\in \s_0(R)$ and $M$ a finitely generated $R$-module, this translates to $\operatorname{G_C-dim}(M)<\infty$.
\end{defn}

The next result is useful for the proof of Theorem~\ref{thm141112b}.

\begin{lem}[\protect{\cite{Sather:scdm}}]\label{lem140724a}
 Assume $A$ is mildly noetherian and let $C\in \s(A)$. For $M\in D(A)$, the following conditions are equivalent.
 \begin{enumerate}
  \item[(a)] $M\simeq 0$,
  \item[(b)] $C\lotimes_A M\simeq 0$, and
  \item[(c)] $\rhom_A(C,M)\simeq 0$.
  \end{enumerate}
\noindent When $M\in D^f_{+}(A)$, these are equivalent to the following:
\begin{enumerate}
  \item[(d)] $\rhom_A(M,C)\simeq 0$.
 \end{enumerate}
\end{lem}

\begin{notation} Let $B,C$ be DG $A$-modules. Then $B\sim C$, if there exists an integer $n$ such that $B\simeq \Sigma^n C$.
\end{notation}

The next result is a DG version of a result of Araya et al. \cite[(5.3)]{takahashi:hiatsb}, see {\cite[Lemma 3.2]{frankild:rrhffd}}.

\begin{lem}\label{lem141113a} Assume $A$ is local and $B$, $C\in \s(A)$. Then $B\approx C$ if and only if $B\simeq \Sigma^n C$ for some integer $n$ in $D(A)$.
\end{lem}

\begin{proof} One implication is straightforward since $C\in \catb_C(A)$ and $B\in \catb_B(A)$.

For the other implication, assume $B\approx C$. Thus $B\in \catb_C(A)$ and $C\in \catb_B(A)$. Thus $B\simeq C\lotimes_A \rhom_A(C,B)$ and $C\simeq B\lotimes_A \rhom_A(B,C)\simeq C\lotimes_A \rhom_A(C,B)\lotimes_A \rhom_A(B,C)$. By Apassov~{\cite[Proposition 2]{apassov:hddgr}}, there exists minimal semifree resolutions $F\xra{\simeq} C$, $G\xra{\simeq} \rhom_A(C,B)$, and $L\xra{\simeq} \rhom_A(B,C)$ over $A$. Since minimal semifree resolutions are unique up to isomorphism, and $F\simeq C \simeq C\lotimes_A \rhom_A(C,B)\lotimes_A \rhom_A(B,C) \simeq F\otimes_A G\otimes_A L$, we have $F\simeq F\otimes_A G\otimes_A L$. By use of the Poincar\'{e} series, it follows that $\rhom_A(C,B)\simeq G\cong \Sigma^n A$ and $L\cong \Sigma^{-n} A$ for some integer $n$. Thus $B\simeq C\otimes_A \Sigma^n A \simeq \Sigma^{n} C$.
\end{proof}

The remainder of this section focuses on $k$-complexes.

\begin{lem}\label{lem141013a} Let $L$ be a $k$-complex. Then $L$ is semiprojective over $k$.
\end{lem}

\begin{proof} Notice that  $L^{\natural}$ is free over $k=k^{\natural}$, therefore projective. Now by {\cite[3.9.1]{avramov:dgha}} $\partial^k=0$ implies $\HH(L)$ and $\operatorname{B}(L)$ are DG $k$-modules. Also we have $\HH(L)$ and $\operatorname{B}(L)$ are projective over $k$. Thus, by {\cite[3.9.7]{avramov:dgha}}, $L$ is semiprojective over $k$. 
\end{proof}

\begin{remark} Let $B'$ and $B''$ be $k$-complexes. By Lemma~\ref{lem141013a}, the complexes $B'$ and $B''$ are semiprojective over $k$. Thus $B'\otimes_k B''\simeq B'\lotimes_k B''$.
\end{remark}

\begin{fact}\label{lem140725a}
 Let $A$, $A'$, $B$, $B'$ be $k$-vector spaces, and let $\alpha:A\to B$ and $\alpha':A'\to B'$ be $k$-module homomorphisms. If $\alpha$ and $\alpha'$ are both isomorphisms, then $\alpha\otimes_k \alpha'$ is an isomorphism. If $A$, $A'$, $B$, $B'\neq 0$, then the converse holds.
\end{fact}

\begin{fact}\label{fact141014a} Let $X'$ and $X''$ be $k$-complexes. Then the K\"unneth formula {\cite[10.81]{rotman:iha}} implies that there is an isomorphism $\bigoplus_{p+q=i} \HH_p(X')\otimes_k \HH_q(X'')\xra{\cong} \HH_i(X'\otimes_k X'')$ given by $\overline{x'}\otimes_k \overline{x''}\mapsto \overline{x'\otimes_k x''}$. Moreover, if $\alpha':A'\to B'$ and $\alpha'':A''\to B''$ are chain maps over $k$, then $\bigoplus_{p+q=i}\HH_p(\alpha')\otimes_k \HH_q(\alpha'')$ is identified with $\HH_i(\alpha'\otimes_k \alpha'')$ under this isomorphism.
\end{fact}

\begin{lem}\label{lem140725b}
 Let $X'$, $X''$, $Y'$, $Y''$ be $k$-complexes, and let $\alpha':X'\to Y'$ and $\alpha'':X''\to Y''$ be chain maps over $k$. If $\alpha'$ and $\alpha''$ are isomorphisms, then $\alpha'\otimes_k \alpha'':X'\otimes_k X''\to Y'\otimes_k Y''$ is an isomorphism. If $X'$, $X''$, $Y'$, $Y''\neq 0$, then the converse holds.
  \end{lem}
  
  \begin{proof} The forward implication is standard. For the reverse implication, assume $X'$, $X''$, $Y'$, $Y''\neq 0$ and $\alpha'\otimes_k \alpha''$ is an isomorphism. By definition we have $(\alpha'\otimes_k \alpha'')_i=\bigoplus_{p+q=i} (\alpha'_p\otimes_k \alpha''_q)$, so $\alpha'_p\otimes_k \alpha''_q$ is an isomorphism for all $p,q$.
 
It remains to show that $\alpha'_p$ and $\alpha''_q$ are isomorphisms for all $p,q$.
Let $X'_{p_0}\neq 0$, $X''_{q_0}\neq 0$, $Y'_{p_1}\neq 0$, and $Y''_{q_1}\neq 0$.
Suppose $X'_{p}=0$. Then $0=X'_{p}\otimes_k X''_{q_1}\xra[\cong]{\alpha'_{p}\otimes_k \alpha_{q_1}} Y'_{p}\otimes_k Y''_{q_1}$. Therefore, $Y'_{p}\otimes_k Y''_{q_1}=0$. Since $Y''_{q_1}\neq 0$, we have $Y'_{p}=0$. So $\alpha'_{p}$ is an isomorphism. By a similar argument $Y'_{p}=0$ implies $\alpha'_{p}$ is an isomorphism. 
Assume $X'_p, Y'_p\neq 0$. The assumption $X''_{q_0}\neq 0$ implies that $Y''_{q_0}\neq 0$. Therefore, $\alpha'_p\otimes_k \alpha''_{q_0}$ isomorphism such that $X'_p, Y'_p, X''_{q_0}, Y''_{q_0}\neq 0$. Thus Lemma~\ref{lem140725a} implies $\alpha'_p$ and $\alpha''_{q_0}$ are isomorphisms.

A symmetric argument shows that $\alpha''_q$ is an isomorphism for all $q$.
\end{proof}

\begin{lem}\label{lem140718d}
 Let $A'$, $B'$, $A''$, $B''$ be $k$-complexes, and let $\alpha':A'\to B'$ and $\alpha'':A''\to B''$ be chain maps over $k$. If $\alpha'$ and $\alpha''$ are quasiisomorphisms, then $\alpha'\otimes_k \alpha''$ is a quasiisomorphism. If $A'$, $B'$, $A''$, $B''\not\simeq 0$, then the converse holds.
\end{lem}

\begin{proof} This follows from Facts~\ref{lem140725a}-\ref{fact141014a} and Lemma~\ref{lem140725b}.
\end{proof}

\begin{lem}\label{lem140724b} Let $M'$ and $M''$ be $k$-complexes. If $M'$ and $M''$ are homologically bounded, then $M'\otimes_k M''$ is homologically bounded. If $M', M''\not\simeq 0$, then the converse holds.
\end{lem}

\begin{proof} \noindent For the forward implication, set $t=\operatorname{sup}(M')$, $w=\operatorname{sup}(M'')$, $s=\operatorname{inf}(M')$, and $l=\operatorname{inf}(M'')$. 

Case 1: If $p+q>t+w$, then $\HH_p(M')\otimes_k \HH_q(M'')=0$ because $p+q>t+w$ implies $p>t$ or $q>w$. Therefore, by Fact~\ref{fact141014a},  $\HH_i(M'\otimes_k M'')=\bigoplus_{p+q=i}(\HH_p(M')\otimes_k \HH_q(M''))=0$ for $i>t+w$.

Case 2: If $p+q<s+l$, then $\HH_p(M')\otimes_k \HH_q(M'')=0$ because $p+q<s+l$ implies $p<s$ or $q<l$. Therefore, $\HH_i(M'\otimes_k M'')=0$ for $i<s+l$.

For the reverse implication suppose $\HH_{p_j}(M')\neq 0$ for infinitely many
indices $j$. Since $M''\not\simeq 0$ there exists an integer $b$ such that $\HH_b(M'')\neq 0$. Now, $\HH_{p_j}(M')\otimes_k \HH_b(M'')\neq 0$ for infinitely many indices $j$. However, $\HH_{p_j}(M')\otimes_k \HH_b(M'')\subset \HH_{p_j+b}(M'\otimes_k M'')$. Hence, there is an infinite number of $\sigma=p_j+b$ such that $\HH_\sigma(M'\otimes_k M'')\neq 0$ which is a contradiction since $M'\otimes_k M''$ is homologically bounded. Thus $\HH_{p_j}(M')\neq 0$ for only finitely many $j$. Hence $M'$ is homologically bounded. By a similar argument $M''$ is homologically bounded. 
\end{proof}

\section{DG Tensor Products} 

This section consists of tools for use in the proofs of our main theorems.

\begin{assumption} In this section $A'$ and $A''$ are DG $R$-algebras and $A:=A'\otimes_R A''$.
\end{assumption}

\begin{lem}\label{lem141028a} Assume that $R=k$ is a field. Let $M'$ and $M''$ be DG $A'$- and $A''$-modules respectively. If $M'$ and $M''$ are degreewise homologically finite over $A'$ and $A''$, respectively, then $M'\otimes_k M''$ is degreewise homologically finite over $A'\otimes_k A''$ under any of the following conditions:
 \begin{enumerate}
  \item[(1)] $M'$ is homologically bounded,
  \item[(2)] $M''$ is homologically bounded,
  \item[(3)] $M'$ and $M''$ are homologically bounded below, or
  \item[(4)] $M'$ and $M''$ are homologically bounded above.
 \end{enumerate}
\end{lem}

\begin{proof}
\noindent (1) By Fact~\ref{fact141014a}, for all $i$ we have $\HH_i(M'\otimes_k M'')\cong \bigoplus_{p+q=i} \HH_p(M')\otimes_k \HH_q(M'')$. Note that this direct sum is finite because $M'\in D_b(A')$.
 
 \vspace{.1in}
 
\noindent Now, $\HH_p(M')$ is finitely generated over $\HH_0(A')$ for all $p$, and $\HH_q(M'')$ is finitely generated over $\HH_0(A'')$ for all $q$, by our assumption. Therefore,
 $\HH_p(M')\otimes_k \HH_q(M'')$ is finitely generated over $\HH_0(A')\otimes_k \HH_0(A'')$ for all $p$ and $q$. Hence $\bigoplus_{p+q=i} \HH_p(M')\otimes_k \HH_q(M'')$ is finitely generated for all $i$.
 
 \noindent The proofs of parts (2)--(4) are similar to proof of part (1). Notice that in each case the assumptions guarantee the direct sum $\bigoplus_{p+q=i} \HH_p(M')\otimes_k \HH_q(M'')$ is finite.
\end{proof}

The next result gives us some flexibility for understanding how DG $A'$- and $A''$-modules yield DG $A$-modules.

\begin{lem}\label{lem141026a} Let $X'$ and $X''$ be DG $A'$- and $A''$-modules respectively. The map $\alpha^{X'}_{X''}:X'\otimes_R X''\to (A\otimes_{A'} X')\otimes_A (A\otimes_{A''} X'')$ given by $x'\otimes x''\mapsto (1\otimes x')\otimes (1\otimes x'')$ is an isomorphism of DG $A$-modules.
\end{lem}

\begin{proof} The given map is the composition of the following sequence of isomorphisms.
\begin{align*} X'\otimes_R X'' &\cong (X'\otimes_{A'} (A'\otimes_R A''))\otimes_{A''} X''\\
&\cong ((A'\otimes_R A'')\otimes_{A'} X')\otimes_{A''} X''\\
&\cong (A\otimes_{A'} X')\otimes_{A''} X''\\
&\cong (A\otimes_{A'} X')\otimes_A (A\otimes_{A''} X'')
\end{align*}

\noindent It is straightforward to show that $\alpha$ is $A$-linear.
\end{proof}

\begin{lem}\label{lem141031a} If $P'$ is a semiprojective DG $A'$-module and $P''$ is semiprojective DG $A''$-module, then $P'\otimes_R P''$ is semiprojective over $A$.
\end{lem}

\begin{proof} By Lemma~\ref{lem141026a}, we have $P'\otimes_R P''\cong (A \otimes_{A'} P')\otimes_A (A\otimes_{A''} P'')$ as DG $A$-modules.
 
 The fact that $P'$ is semiprojective over $A'$ implies that $A\otimes_{A'} P'$ is semiprojective over $A$ because $$\hom_A(A\otimes_{A'} P',-)\cong \hom_{A'}(P',\hom_A(A,-))\cong \hom_{A'}(P',-).$$ Similarly, $A\otimes_{A''} P''$ is semiprojective over $A$. 
 Now $X,Y$ semiprojective over $A$ implies $X\otimes_A Y$ is semiprojective over $A$ because $\hom_A(X\otimes_A Y,-)\cong \hom_A(Y,\hom_A(X,-))$. Therefore,
 $A\otimes_{A'} P'$ semiprojective over $A$ and $A\otimes_{A''} P''$ semiprojective over $A$ imply that $(A\otimes_{A'} P')\otimes_A (A\otimes_{A''} P'')\cong P'\otimes_R P''$ is semiprojective. 
\end{proof}

\begin{lem}\label{lem140718a} Assume that $R=k$ is a field. Let $M'$ and $M''$ be DG $A'$- and $A''$-modules respectively.
If $P'\xra[\simeq]{\alpha'} M'$ and $P''\xra[\simeq]{\alpha''} M''$ are semiprojective resolutions over $A'$ and $A''$, respectively, then
 $P'\otimes_k P''\xra[\simeq]{\alpha'\otimes_k \alpha''}M'\otimes_k M''$ is a semiprojective resolution over $A$. 
\end{lem}

\begin{proof}  Notice $P'\otimes_k P''$ is semiprojective over $A$ and $P'\otimes_k P''\xra{\alpha'\otimes_k \alpha''}M'\otimes_k M''$ is a quasiisomorphism by Lemmas~\ref{lem141031a} and \ref{lem140718d}.
\end{proof}

Our next result is similar in flavor to Lemma~\ref{lem141026}.

\begin{lem}\label{lem140929a} Let $X'$ and $X''$ be $A'$- and $A''$-modules, respectively. The map $$\tilde{\gamma}^{X',X''}_{Y',Y''}:(X'\otimes_{A'}Y')\otimes_R (X''\otimes_{A''} Y'')\to (X'\otimes_R X'')\otimes_A(Y'\otimes_R Y'')$$ given by $(x'\otimes y')\otimes (x''\otimes y'')\mapsto (-1)^{|y'||x''|}(x'\otimes x'')\otimes (y'\otimes y'')$ is an isomorphism of DG $A$-modules. 
\end{lem}

\begin{proof} Lemma~\ref{lem141026a} gives the first and last isomorphisms in the following display. The second and third isomorphisms are by associativity, commutativity, etc. of tensor products.
\begin{align*} (X'\otimes_R X'')\otimes_A (Y'\otimes_R Y'') & \cong [(A\otimes_{A'} X')\otimes_A (A\otimes_{A''} X'')]\otimes_A [(A\otimes_{A'} Y')\otimes_A (A\otimes_{A''} Y'')]\\
&\cong (A\otimes_{A'} X')\otimes_A (A\otimes_{A'} Y')\otimes_A (A\otimes_{A''} X'')\otimes_A (A\otimes_{A''} Y'')\\
&\cong (A\otimes_{A'}(X'\otimes_{A'} Y'))\otimes_A (A\otimes_{A''}(X''\otimes_{A''}Y''))\\
&\cong (X'\otimes_{A'} Y')\otimes_R (X''\otimes_{A''} Y'')
\end{align*}

\noindent It is straightforward to show that $\tilde{\gamma}^{X',X''}_{Y',Y''}$ is the composition of the displayed isomorphisms and is $A$-linear.
\end{proof}

\begin{lem}\label{lem141028c} Assume that $R=k$ is a field. Then the morphism $$\gamma^{X',X''}_{Y',Y''}:(X'\lotimes_{A'} Y')\otimes_k (X''\lotimes_{A''} Y'') \to (X'\otimes_k X'')\lotimes_A (Y'\otimes_k Y'')$$ induced by the morphism $\tilde{\gamma}^{P',P''}_{Q',Q''}$ from Lemma~\ref{lem140929a} is an isomorphism in $D(A)$.
\end{lem}

\begin{proof} Let $P'\xra{\simeq} X'$, $P''\xra{\simeq} X''$, $Q'\xra{\simeq} Y'$, and $Q''\xra{\simeq} Y''$ be semiprojective resolutions over $A'$ and $A''$ as appropriate. By Lemma~\ref{lem140929a}, the map $$\tilde{\gamma}^{P',P''}_{Q',Q''}:(P'\otimes_{A'} Q')\otimes_k (P''\otimes_{A''} Q'')\to (P'\otimes_k P'')\otimes_A (Q'\otimes_k Q'')$$ is an isomorphism of DG $A$-modules. Therefore, $\gamma^{X',X''}_{Y',Y''}$ is an isomorphism in $D(A)$.
\end{proof}

The remainder of this section is devoted to understanding $\rhom_A(N,M)$ for DG $A$-modules $M$ and $N$ constructed as above.

\begin{defn}\label{lem140924b} Let $N',M'$ and $N'',M''$ be DG $A'$- and $A''$-modules, respectively. Consider elements $f'\in \hom_{A'}(N',M')$ and $f''\in \hom_{A''}(N'',M'')$. 
 Let $f'\boxtimes f'': N'\otimes_R N''\to M'\otimes_R M''$ be given by $(f'\boxtimes f'')_{|x'\otimes x''|}(x'\otimes x'')=(-1)^{|f''||x'|}f'_{|x'|}(x')\otimes f''_{|x''|}(x'')$. 
\end{defn}

\begin{remark}\label{remark141031b} With notation as in Definition~\ref{lem140924b}, the map $f'\boxtimes f''$ is well-defined and $A$-linear.
\end{remark}

\begin{ex}\label{ex141026b} Let $X'$ and $X''$ be $R$-complexes. Then we have $\partial^{X'\otimes_R X''}=(\partial^{X'}\boxtimes \text{id})+(\text{id}\boxtimes \partial^{X''})$.
\end{ex}

\begin{defn}\label{defn141031c}  Let $N',M'$ and $N'',M''$ be DG $A'$- and $A''$-modules, respectively. Let  $$\tilde{\eta}^{N',N''}_{M',M''}:\hom_{A'}(N',M')\otimes_R \hom_{A''}(N'',M'')\to \hom_A(N'\otimes_R N'',M'\otimes_R M'')$$ be given by $f'\otimes f''\mapsto f'\boxtimes f''$.
\end{defn}

\begin{remark}\label{lem140924a} The map $\tilde{\eta}^{N',N''}_{M',M''}$ is a well-defined morphism of DG $A$-modules.
\end{remark}

\begin{prop}\label{lem140924c} Assume that $R=k$ is a field. If $N',N''$ are degreewise finite, semifree, bounded below DG $A'$- and $A''$-modules, respectively, and $M',M''$ are bounded above DG $A'$- and $A''$-modules, respectively, then the morphism $\tilde{\eta}^{N',N''}_{M',M''}$ is an isomorphism of DG $A$-modules.
\end{prop}

\begin{proof} It suffices to show that the morphism $$\tilde{\eta}^{N',N''}_{M',M''}:\hom_{A'^{\natural}}(N'^{\natural},M'^{\natural})\otimes_k \hom_{A''^{\natural}}(N''^{\natural},M''^{\natural})\to \hom_{A^{\natural}}((N'\otimes_k N'')^{\natural}, (M'\otimes_k M'')^{\natural})$$ is an isomorphism. Therefore, without loss of generality, assume that all differentials are $0$. Thus $N'\cong \bigoplus_{p\geq p_0} \Sigma^p (A')^{\beta'_p}$ and $N''\cong \bigoplus_{q\geq q_0} \Sigma^q (A'')^{\beta''_q}$ for some integers $\beta'_p$, $\beta'_q\geq 0$.

Special case: Assume $N'=A'$ and $N''=A''$. Set $\tilde{\eta}=\tilde{\eta}^{A',A''}_{M',M''}$. It is straightforward to show that the following diagram commutes.  $$\xymatrix{\hom_{A'}(A',M')\otimes_k \hom_{A''}(A'',M'') \ar[d]_-{\tilde{\eta}} \ar[r]_-{\cong} & M'\otimes_k M''\\ 
 \hom_A(A,M'\otimes_k M'')  \ar[ur]_{\cong} &}$$.
 
\noindent Hence $\tilde{\eta}$ is an isomorphism in this case.

General case: Set $\tilde{\eta}'=\tilde{\eta}^{N',N''}_{M',M''}$. First we have \begin{align*} N'\otimes_k N'' \cong& \bigoplus_{p\geq p_0} \bigoplus_{q\geq q_0} \Sigma^{p+q} (A'\otimes_k A'')^{\beta'_p \beta''_q}.
\end{align*}

\noindent Now, for all $m\in \mathbb{Z}$, our boundedness condition on $M'$ implies that
    
\begin{align*} \hom_{A'}(N',M')_m \cong& \hom_{A'}\left(\bigoplus_{p\geq p_0} \Sigma^p (A')^{\beta'_p}, M'\right)_m\\
\cong& \prod_{p\geq p_0} \hom_{A'}(\Sigma^p (A')^{\beta'_p}, M')_m\\
=& \bigoplus_{p\geq p_0} \hom_{A'} (\Sigma^p (A')^{\beta_p}, M')_m.
\end{align*}

\noindent Similarly, for all $n\in \mathbb{Z}$, we have

\begin{align*}\hom_{A''}(N'',M'')_n \cong& \hom_{A''}\left(\displaystyle{\bigoplus_{q>q_0}}\Sigma^{q} (A'')^{\beta''_q}, M''\right)_n\\
\cong& \displaystyle{\bigoplus_{q\geq q_0} \hom_{A''}(\Sigma^q (A'')^{\beta''_q}, M'')_n}.
\end{align*}

\noindent The domain of $\tilde{\eta}'_i$ decomposes as follows.

\begin{align*} [\hom_{A'}(N',M')\otimes_k \hom_{A''}(N'',M'')]_i \hspace{-3cm}\\
\cong& \bigoplus_{m+n=i}\left[\hom_{A'}\left(\bigoplus_{p\geq p_0} \Sigma^p (A')^{\beta'_p}, M'\right)_m\otimes_k \hom_{A''}\left(\bigoplus_{q\geq q_0} \Sigma^q (A'')^{\beta''_q}, M''\right)_n\right] \\
\cong& \bigoplus_{m+n=i} \bigoplus_{p\geq p_0} \bigoplus_{q\geq q_0} \left[\hom_{A'}(\Sigma^p A',M)^{\beta'_p}_m\otimes_k \hom_{A''}(\Sigma^q A'', M'')^{\beta''_q}_n\right]\\
\cong& \bigoplus_{m+n=i} \bigoplus_{p\geq p_0} \bigoplus_{q\geq q_0} \Sigma^{-p-q} [\hom_{A'}(A',M)_m\otimes_k  \hom_{A''}( A'', M'')_n]^{\beta'_p \beta''_q}\\
\cong&  \bigoplus_{p\geq p_0} \bigoplus_{q\geq q_0} \Sigma^{-p-q} \left[\bigoplus_{m+n=i} \hom_{A'}(A',M)_m\otimes_k  \hom_{A''}( A'', M'')_n\right]^{\beta'_p \beta''_q}\\
\cong& \bigoplus_{p\geq p_0} \bigoplus_{q\geq q_0} \Sigma^{-p-q} [\hom_{A'}(A',M)\otimes_k  \hom_{A''}( A'', M'')]_i^{\beta'_p \beta''_q}
     \end{align*}

\noindent Next, we consider the codomain in degree $i$. \begin{align*} \hom_A(N'\otimes_k N'', M'\otimes_k M'')_i \cong& \hom_A\left(\left(\bigoplus_{p\geq p_0} \Sigma^p (A')^{\beta'_p}\right)\otimes_k \left(\bigoplus_{q\geq q_0} \Sigma^q (A'')^{\beta''_q}\right), M'\otimes_k M''\right)_i\\
\cong& \hom_A\left(\bigoplus_{p\geq p_0} \bigoplus_{q\geq q_0} \Sigma^{p+q} (A'\otimes_k A'')^{\beta'_p \beta''_q}, M'\otimes_k M''\right)_i\\
\cong& \bigoplus_{p\geq p_0} \bigoplus_{q\geq q_0} \hom_A(\Sigma^{p+q} (A'\otimes_k A'')^{\beta'_p \beta''_q}, M'\otimes_k M'')_i\\
\cong& \bigoplus_{p\geq p_0} \bigoplus_{q\geq q_0} \Sigma^{-p-q} \hom_A(A'\otimes_k A'', M'\otimes_k M'')_i^{\beta'_p \beta''_q}  
\end{align*}

\noindent It is straightforward to show that $\tilde{\eta}$ is compatible with direct sums and shifts. Therefore, we have $\tilde{\eta}'=\bigoplus_{p\geq p_0} \bigoplus_{q\geq q_0} \Sigma^{-p-q}\tilde{\eta}.$ Since $\tilde{\eta}$ is an isomorphism by our special case, we conclude that $\tilde{\eta}'$ is an isomorphism.
\end{proof}

\begin{remark} Assume that $R=k$ is a field. Let $N'$ and $N''$ be DG $A'$- and $A''$-modules respectively. Let $P'\xra{\simeq} N'$ and $P''\xra{\simeq} N''$ be semiprojective resolutions over $A'$ and $A"$, respectively. By Lemma~\ref{lem140718a}, we have that $P'\otimes_k P'' \xra{\simeq} N'\otimes_k N''$ is a semiprojective resolution over $A$. Therefore, $\tilde{\eta}^{P',P''}_{M',M''}:\hom_{A'}(P',M')\otimes_k \hom_{A''}(P'',M'')\to \hom_A(P'\otimes_k P'',M'\otimes_k M'')$ represents a well-defined morphism $\eta^{N',N''}_{M',M''}: \rhom_{A'}(N',M')\otimes_k \rhom_{A''}(N'',M'')\to \rhom_A(N'\otimes_k N'',M'\otimes_k M'')$ in $D(A)$.
\end{remark}

For the nest result, notice if $A'$ and $A''$ are weakly noetherian, then DG modules $N'\in D_{+}^f(A')$ and $N''\in D_{+}^f(A'')$ admit degreewise finite semifree resolutions by Fact~\ref{fact141013c}.

\begin{prop}\label{lem140924g} Assume that $R=k$ is a field. Let $N'\in D^f_{+}(A')$ and $N''\in D^f_{+}(A'')$ admit degreewise finite semifree resolutions over $A'$ and $A''$, respectively, and $M'\in D_{-}(A')$ and $M''\in D_{-}(A'')$. Then $\eta^{N',N''}_{M',M''}$ is an isomorphism in $D(A)$.
\end{prop}

\begin{proof} Notice that $M'$ and $M''$ homologically bounded above implies there exists $L'$ and $L''$ such that $\alpha':M'\xra{\cong} L'$ and $\alpha'':M''\xra{\cong} L''$ where $L'$ and $L''$ are bounded above. Therefore, we can replace $M'$ and $M''$ by $L'$ and $L''$ to assume that $M'$ and $M''$ are bounded above. By assumption, there exist semifree resolutions $P'\xra{\simeq} N'$ and $P''\xra{\simeq} N''$ such that $P',P''$ are bounded below and degreewise finite. Therefore, we can replace $N'$ and $N''$ by $P'$ and $P''$ respectively to assume that $N'$ and $N''$ are semifree, bounded below, and degreewise finite. The result now follows from Lemma~\ref{lem140924c}.
\end{proof}

\section{Semidualizing DG Modules} 

In this section we prove the main result of this paper and document a few corollaries. 

\begin{assumption}\label{assump141124a} In this section $k$ is a field, $A'$ and $A''$ are DG $k$-algebras such that $A'\not\simeq 0\not\simeq A''$, and $A:=A'\otimes_k A''$.
\end{assumption}

The next two results are the keys for proving Theorem~\ref{thm141112b} from the introduction.

\begin{thm}\label{thm141016a} If $M'$ and $M''$ are semidualizing over $A'$ and $A''$, respectively, then $M'\otimes_k M''$ is semidualizing over $A$. If $A'$ and $A''$ are mildly noetherian and $M'\in D^f_b(A')$, $M''\in D^f_b(A'')$, then the converse holds.
  \end{thm}
  
  \begin{proof} Step 1. Note that $A',A''\not\simeq 0$, by Assumption~\ref{assump141124a}. Thus we have $A\not\simeq 0$, e.g., by the K\"unneth formula. 
  
  Step 2. If $M'\in \s(A')$, then $M'\not\simeq 0$ because $\rhom_{A'}(M',M')\simeq A'\not\simeq 0$. On the other hand, if $M'\otimes_k M''\in \s(A)$, then $M'\otimes_k M''\not\simeq 0$, so $M'\not\simeq 0$. Thus, we assume for the remainder of the proof that $M'\not\simeq 0$ and similarly, $M''\not\simeq 0$. 
  
  Step 3. In the forward implication we assume $M'\in \s(A')$ and $M''\in \s(A'')$, therefore we have $M'\in D_b^f(A')$ and $M''\in D_b^f(A'')$. Thus, we assume for the remainder of the proof that $M'\in D_b^f(A')$ and $M''\in D_b^f(A'')$. 
  
  Step 4. We assume for the remainder of the proof that $M'$ and $M''$ admit degreewise finite semifree resolutions. Notice, in the forward implication, the conditions $M'\in \s(A')$ and $M''\in \s(A'')$ guarantee that such resolutions exist; in the reverse implication, since $A'$ and $A''$ are weakly noetherian and $M'\in D^f_b(A')$, $M''\in D^f_b(A'')$, Fact~\ref{fact141013c}(b) guarantees that such resolutions exist. Note that it follows that the DG module $M'\otimes_k M''\in D_b^f(A)$ has such a resolution over $A$; see Lemmas~\ref{lem140724b}(a) and \ref{lem141028a}..
  
  Step 5: Consider the following commutative diagram in $D(A)$.
  
$$\xymatrix@C=2cm{A=A'\otimes_k A'' \ar[r]^-{\chi_{M'}^{A'} \otimes_k \chi_{M''}^{A''}} \ar[dr]_<<<<<<<<<<<<<<<{\chi_{M'\otimes_k M''}^{A}} & \mathbf{R}\Hom_{A'}(M',M') \otimes_k \mathbf{R}\Hom_{A''}(M'',M'') \ar[d]^{\eta^{M',M''}_{N',N''}}_{\simeq}\\
& \mathbf{R}\Hom_{A}(M'\otimes_k M'',M'\otimes_k M'') }$$

\noindent Notice that the morphism $\eta^{M',M''}_{N',N''}$ in this diagram is an isomorphism by Lemma~\ref{lem140924g}.

In the forward implication, the morphism $\chi_{M'}^{A'}$ is an isomorphism in $D(A')$ and $\chi_{M''}^{A''}$ is an isomorphism in $D(A'')$, so $\chi_{M'}^{A'}\otimes_k \chi_{M''}^{A''}$ is an isomorphism in $D(A)$ by Lemma~\ref{lem140718d}. Therefore, the commutative diagram implies that $\chi_{M'\otimes_k M''}^{A}$ is an isomorphism in $D(A)$.

In the reverse implication, our commutative diagram with $\eta^{M',M''}_{N',N''}$ and $\chi_{M'\otimes_k M''}^{A}$ isomorphisms in $D(A)$ imply that $\chi_{M'}^{A'}\otimes_k \chi_{M''}^{A''}$ is an isomorphism in $D(A)$. In particular, we have $$\rhom_{A'}(M',M')\otimes_k \rhom_{A''}(M'',M'')\simeq A\not\simeq 0$$ so $\rhom_{A'}(M',M'),\rhom_{A''}(M'',M'')\not\simeq 0$. Thus Lemma~\ref{lem140724a} and Lemma~\ref{lem140718d} imply that $\chi_{M'}^{A'}$ is an isomorphism in $D(A')$ and $\chi_{M''}^{A''}$ is an isomorphism in $D(A'')$. 
 \end{proof}
  
\begin{thm}\label{lem140922a} Fix $M'\in \s(A')$ and $M''\in \s(A'')$, and let $N'\in D(A')$ and $N''\in D(A'')$. If $N'\in \catb_{M'}(A')$ and $N''\in \catb_{M''}(A'')$, then $N'\otimes_k N''\in \catb_{M'\otimes_k M''}(A)$. If $A'$ and $A''$ are mildly noetherian and $N'$, $N''\not\simeq 0$, then the converse holds.
\end{thm}

\begin{proof} Step 1: If $N'\simeq 0$ or $N''\simeq 0$, then $N'\otimes_k N'' \simeq 0\in \catbc(A)$. Therefore, assume for the rest of the proof that $N', N''\not\simeq 0$.

Step 2: By Lemma~\ref{lem140724b} we have $N'\in D_b(A')$ and $N''\in D_b(A'')$ if and only if $N'\otimes_k N''\in D_b(A)$. Therefore, assume for the rest of the proof that $N'\in D_b(A')$ and $N''\in D_b(A'')$.

Step 3: We show that $\rhom_{A'}(M',N')\in D_b(A')$ and $\rhom_{A''}(M'',N'')\in D_b(A'')$ if and only if $\rhom_{A}(M'\otimes_k M'', N'\otimes_k N'')\in D_b(A)$. Notice, by Lemma~\ref{lem140924g} we have $$\rhom_{A}(M'\otimes_k M'', N'\otimes_k N'')\simeq \rhom_{A'}(M',N')\otimes_k \rhom_{A''}(M'',N'')$$ in $D(A)$. Now, by Lemma~\ref{lem140724a}, since $M'\in \s(A')$ and $N'\not\simeq 0$ we have $\rhom_{A'}(M',N')\not\simeq 0$. Similarly, $\rhom_{A''}(M'',N'')\not\simeq 0$. Therefore, by Lemma~\ref{lem140724b} parts (a) and (b) we have $\rhom_{A'}(M',N')\in D_b(A')$ and $\rhom_{A''}(M'',N'')\in D_b(A'')$ if and only if $\rhom_{A}(M'\otimes_k M'', N'\otimes_k N'')\in D_b(A)$.

Therefore, assume for the rest of the proof $\rhom_{A'}(M',N')\in D_b(A')$ and $\rhom_{A''}(M'',N'')\in D_b(A'')$.

Step 4: We need to show that $\xi_{N'\otimes_k N''}^{M'\otimes_k M''}$ is an isomorphism in $D(A)$ if and only if $\xi_{N'}^{M'}$ and $\xi_{N''}^{M''}$ are isomorphisms in $D(A')$ and $D(A'')$, respectively.
Consider the following commutative diagram in $D(A)$.

$$\xymatrix@C=8mm{(M'\lotimes_{A'}\rhom_{A'}(M',N'))\otimes_k (M''\lotimes_{A''}\rhom_{A''}(M'',N'')) \ar[d]^-{\simeq}_{\gamma^{M',M''}_{\rhom_{A'}(M',N'),\rhom_{A''}(M'',N'')}} \ar[dr]^>>>>>>>>>>>>{\xi_{N'}^{M'}\otimes_k \xi_{N''}^{M''}}  & \\
 (M'\otimes_k M'')\lotimes_A(\rhom_{A'}(M',N')\otimes_k \rhom_{A''}(M'',N'')) \ar[d]^-{\simeq}_-{(M'\otimes_k M'')\lotimes_A \eta^{M',M''}_{N',N''}} & N'\otimes_k N''\\
 (M'\otimes_k M'')\lotimes_A \rhom_A(M'\otimes_k M'',N'\otimes_k N'') \ar[ur]_>>>>>>>>>>>>>>{\xi_{N'\otimes_k N''}^{M'\otimes_k M''}} }$$
 
\noindent Notice that $\gamma^{M',M''}_{\rhom_{A'}(M',N'),\rhom_{A''}(M'',N'')}$ and $(N'\otimes_k N'')\lotimes_A \eta^{M',M''}_{N',N''}$ are isomorphisms by Lemmas~\ref{lem141028c} and \ref{lem140924g}. Thus we have that $\xi_{N'\lotimes_ k N''}^{M'\lotimes_k M''}$ is an isomorphism if and only if $\xi_{N'}^{M'}\otimes_k \xi_{N''}^{M''}$ is an isomorphism, that is, if and only if $\xi_{N'}^{M'}$ and $\xi_{N''}^{M''}$ are isomorphisms by Lemma~\ref{lem140718d} and Lemma~\ref{lem140724a}. (Note that this uses the following: by Lemma~\ref{lem140724a}, since $N'\not\simeq 0$ and $M'\in \s(A')$ we have $\rhom_{A'}(M',N')\not\simeq 0$ and furthermore, $M'\lotimes_{A'} \rhom_{A'}(M',N')\not\simeq 0$.) 
 \end{proof}
 
The next two results are proved similarly to Theorem~\ref{lem140922a}.
 
\begin{thm} Fix $M'\in \s(A')$ and $M''\in \s(A'')$ and let $N'\in D(A')$ and $N''\in D(A'')$. If $N'\in \cata_{M'}(A')$ and $N''\in \cata_{M''}(A'')$, then $N'\otimes_k N''\in \cata_{M'\otimes_k M''}(A)$. If $A'$ and $A''$ are mildly noetherian and $N'\not\simeq 0$ and $N''\not\simeq 0$, then the converse holds.
\end{thm}

 \begin{thm} Fix $M'\in \s(A')$ and $M''\in \s(A'')$ and let $N'\in D_b^f(A')$ and $N''\in D_b^f(A'')$. If $N'$ is derived $M'$ reflexive over $A'$ and $N''$ is derived $M''$ reflexive over $A''$, then $N'\otimes_k N''$ is derived $M'\otimes_k M''$ reflexive over $A$. If $A'$ and $A''$ are mildly noetherian and $N'\not\simeq 0$ and $N''\not\simeq 0$, then the converse holds.
\end{thm}

In the next result, we use the notation of \ref{notation141115a}.

 \begin{thm}\label{lem140718e} Assume $M',N'\in \s(A')$ and $M'',N''\in \s(A'')$. If $M'\approx N'$ and $M''\approx N''$, then $M'\otimes_k M''\approx N'\otimes_k N''$. If $A'$ and $A''$ are mildly noetherian, then the converse holds.
\end{thm}

\begin{proof}
 By a symmetric argument, this is a consequence of Theorem~\ref{lem140922a}.
\end{proof}

\begin{thm}\label{thm141112a} Assume $A'$ and $A''$ are mildly noetherian. Then the map $\psi:\overline{\s}(A')\times \overline{\s}(A'')\to \overline{\s}(A)$ given by $\psi(C',C'')=C'\otimes_k C''$ is well-defined and injective.
\end{thm}

\begin{proof} This follows from Theorem~\ref{thm141016a} with Theorem~\ref{lem140718e}. For instance, assume $\psi(M',M'')=\psi(N',N'')$. Then $M'\otimes_k M''\approx N'\otimes_k N''$. Thus, $M'\approx N'$ and $M''\approx N''$ by Theorem~\ref{lem140718e}. 
\end{proof}

\begin{para}[Proof of Theorem~\ref{thm141112b}]\label{proof141021a} (a): This follows from Theorem~\ref{thm141016a}.

(b): The map $\psi$ being well-defined is due to part (a). The map $\psi$ being injective is a special case of Theorem~\ref{thm141112a} due to Lemma~\ref{lem141113a}. 
\qed
\end{para}

 We conclude by documenting some special cases of the above results.
 
\begin{cor}\label{cor140731a}
 Let $R_i$ be a local $k$-algebra for $i=1,2$. Let $X_i$ be a finitely generated $R_i$-module for $i=1,2$. 
  \enumerate
  \item One has $X_1\otimes_k X_2\in \s_0(R_1\otimes_k R_2)$ if and only if $X_i\in \mathfrak{S}_0(R_i)$ for $i=1,2$.
  \item The map $\psi:\mathfrak{S}_0(R_1)\times \mathfrak{S}_0(R_2)\rightarrow \mathfrak{S}_0(R_1\otimes_k R_2)$ given by $\psi(C_1,C_2)=C_1\otimes_k C_2$ is well-defined and injective.
\end{cor}

\begin{cor}\label{cor140731b}
 Let $R_i$ be a local $k$-algebra for $i=1,2$.  
  Let $X_i\in D^f_b(R_i)$ for $i=1,2$.
  \enumerate
  \item One has $X_1\otimes_k X_2\in \mathfrak{S}(R_1\otimes_k R_2)$ if and only if $X_i\in \mathfrak{S}(R_i)$ for $i=1,2$.
  \item The map $\psi:\mathfrak{S}(R_1)\times \mathfrak{S}(R_2)\rightarrow \mathfrak{S}(R_1\otimes_k R_2)$ given by $\psi(C_1,C_2)=C_1\otimes_k C_2$ is well-defined and injective.
\end{cor}

\begin{cor}\label{cor140731c}
 Let $R_i$ be a $k$-algebra for $i=1,2$. Let $X_i\in D^f_b(R_i)$ for $i=1,2$. Then the map $\psi:\overline{\mathfrak{S}}(R_1)\times \overline{\mathfrak{S}}(R_2)\rightarrow \overline{\mathfrak{S}}(R_1\otimes_k R_2)$ given by $\psi(C_1,C_2)=C_1\otimes_k C_2$ is well-defined and injective.
\end{cor}


\begin{thebibliography}{10}

\bibitem{apassov:hddgr}
D.~Apassov.
\newblock Homological dimensions over differential graded rings.
\newblock In {\em Complexes and Differential Graded Modules}, Ph.D. thesis,
  pages 25--39. Lund University, 1999.

\bibitem{takahashi:hiatsb}
T.\ Araya, R.\ Takahashi, and Y.\ Yoshino.
\newblock Homological invariants associated to semi-dualizing bimodules.
\newblock {\em J. Math. Kyoto Univ.}, 45(2):287--306, 2005.

\bibitem{avramov:ifr}
L.~L. Avramov.
\newblock Infinite free resolutions.
\newblock In {\em Six lectures on commutative algebra (Bellaterra, 1996)},
  volume 166 of {\em Progr. Math.}, pages 1--118. Birkh\"auser, Basel, 1998.

\bibitem{avramov:dgha}
L.~L. Avramov, H.-B.\ Foxby, and S.\ Halperin.
\newblock Differential graded homological algebra.
\newblock in preparation.

\bibitem{beck:kddga}
Kristen~A. Beck and Sean Sather-Wagstaff.
\newblock Krull dimension for differential graded algebras.
\newblock {\em Arch. Math. (Basel)}, 101(2):111--119, 2013.

\bibitem{beck:sgidgca}
Kristen~A. Beck and Sean Sather-Wagstaff.
\newblock A somewhat gentle introduction to differential graded commutative
  algebra.
\newblock In Susan~M. Cooper and Sean Sather-Wagstaff, editors, {\em
  Connections Between Algebra, Combinatorics, and Geometry}, volume~76 of {\em
  Springer Proceedings in Mathematics \& Statistics}, pages 3--99. Springer New
  York, 2014.

\bibitem{christensen:scatac}
L.~W. Christensen.
\newblock Semi-dualizing complexes and their {A}uslander categories.
\newblock {\em Trans. Amer. Math. Soc.}, 353(5):1839--1883, 2001.

\bibitem{christensen:dvke}
L.~W. Christensen and S.\ Sather-Wagstaff.
\newblock Descent via {K}oszul extensions.
\newblock {\em J. Algebra}, 322(9):3026--3046, 2009.

\bibitem{foxby:gmarm}
H.-B.\ Foxby.
\newblock Gorenstein modules and related modules.
\newblock {\em Math. Scand.}, 31:267--284 (1973), 1972.

\bibitem{foxby:gdcmr}
H.-B.\ Foxby.
\newblock Gorenstein dimensions over {C}ohen-{M}acaulay rings.
\newblock In {\em Proceedings of the international conference on commutative
  algebra (W.\ Bruns, ed.)}, pages 59--63. Universit\"{a}t Osnabr\"{u}ck, 1994.

\bibitem{frankild:rrhffd}
A.\ Frankild and S.\ Sather-Wagstaff.
\newblock Reflexivity and ring homomorphisms of finite flat dimension.
\newblock {\em Comm. Algebra}, 35(2):461--500, 2007.

\bibitem{golod:gdagpi}
E.~S. Golod.
\newblock {$G$}-dimension and generalized perfect ideals.
\newblock {\em Trudy Mat. Inst. Steklov.}, 165:62--66, 1984.
\newblock Algebraic geometry and its applications.

\bibitem{nasseh:lrfsdc}
S.~Nasseh and S.~Sather-Wagstaff.
\newblock A local ring has only finitely many semidualizing complexes up to
  shift-isomorphism.
\newblock preprint (2012) \texttt{arXiv:1201.0037}.

\bibitem{rotman:iha}
J.~J. Rotman.
\newblock {\em An introduction to homological algebra}, volume~85 of {\em Pure
  and Applied Mathematics}.
\newblock Academic Press Inc., New York, 1979.

\bibitem{Sather:scdm}
S.~Sather-Wagstaff.
\newblock Support and co-support for {DG} modules.
\newblock in preparation.

\bibitem{vasconcelos:dtmc}
W.~V. Vasconcelos.
\newblock {\em Divisor theory in module categories}.
\newblock North-Holland Publishing Co., Amsterdam, 1974.
\newblock North-Holland Mathematics Studies, No. 14, Notas de Matem\'atica No.
  53. [Notes on Mathematics, No. 53].

\end{thebibliography}

\end{document}